\documentclass[11pt,leqno]{amsart}
\usepackage{amssymb,amsfonts}
\usepackage{amsmath,amsthm,amsxtra}
\usepackage[all]{xy}
\usepackage{color}
\usepackage{verbatim}

\usepackage{enumerate}

\newcommand\bigcheck[1]{#1 \raise1ex\hbox{$\hspace{-1ex}{}^\vee$}}
\newcommand\sucheck[1]{#1 \raise0.5ex\hbox{$\hspace{-1ex}{}^\vee$}}

\newtheorem{theorem}{Theorem}[section]
\newtheorem{lemma}[theorem]{Lemma}
\newtheorem{corollary}[theorem]{Corollary}
\newtheorem{proposition}[theorem]{Proposition}
\newtheorem*{lemma*}{Lemma}

\theoremstyle{definition}
\newtheorem{definition}[theorem]{Definition}

\theoremstyle{remark}
\newtheorem{remark}[theorem]{Remark}
\newtheorem{example}[theorem]{Example}

\newtheorem{conjecture}[theorem]{Conjecture}

\newcommand{\mc}[1]{{\mathcal #1}}

\newcommand{\mb}[1]{{\mathbb #1}}

\renewcommand{\tilde}{\widetilde}

\newcommand{\Mat}{\mathop{\rm Mat }}

\newcommand{\ord}{\mathop{\rm ord }}
\newcommand{\tord}{\mathop{\rm tord }}

\renewcommand{\d}{\mathop{\rm d }}
\newcommand{\dd}{\mathop{\rm dd }}

\definecolor{light}{gray}{.9}

\newcommand{\pecettaold}[1]{}

\begin{document}


\title{Some algebraic properties of differential operators}

\author{Sylvain Carpentier}
\address{S. Carpentier, Ecole Normale Superieure, Paris, France}
\email{scarpent@clipper.ens.fr}
\thanks{S. Carpentier supported in part by Department of Mathematics, M.I.T.}

\author{Alberto De Sole}
\address{A. De Sole, Dip. di Matematica, ``Sapienza'' Universit\`a di Roma, Italy }
\email{desole@mat.uniroma1.it}
\thanks{A. De Sole supported in part by Department of Mathematics, M.I.T.}

\author{Victor G. Kac}
\address{V. Kac, Department of Mathematics, M.I.T., Cambridge, MA 02139, USA.}
\email{kac@math.mit.edu}
\thanks{V. Kac supported in part by an NSF grant and by the Center of Mathematics and Theoretical Physics in Rome}

\date{}


\begin{abstract}
\noindent 
First, we study the subskewfield of \emph{rational} pseudodifferential operators 
over a differential field $\mc K$
generated in the skewfield $\mc K((\partial^{-1}))$ of pseudodifferential operators over $\mc K$
by the subalgebra $\mc K[\partial]$ of all differential operators.
Second, we show that the Dieudonn\`e determinant of a matrix pseudodifferential operator
with coefficients in a differential subring $\mc A$ of $\mc K$
lies in the integral closure of $\mc A$ in $\mc K$,
and we give an example of a $2\times 2$ matrix with entries in $\mc A[\partial]$
whose Dieudonn\`e determiant does not lie in $\mc A$.
\end{abstract}

\maketitle


\section{Introduction}
\label{sec:intro}

Let $\mc K$ be a differential field with derivation $\partial$ and let $\mc K[\partial]$
be the algebra of differential operators over $\mc K$.
First, we recall the well known fact that the ring $\mc K[\partial]$ is left and right Euclidean,
hence it satisfies the left and right Ore conditions.
Consequently, we may consider its skewfield of fractions $\mc K(\partial)$,
called the skewfield of rational pseudodifferential operators.
It follows from the Ore Theorem (see e.g. \cite{Art99}) that any rational pseudodifferential operator $R$
can be represented as a right (resp. left) fraction $AS^{-1}$ (resp. $S_1^{-1}A_1$),
where $A,A_1,S,S_1\in\mc K[\partial]$.
We show that these fractions have a unique representations in ``lowest terms''.
Namely if $S$ (resp. $S_1$) has minimal possible order and is monic,
then any other right (resp. left) representation of $R$ can be obtained by multiplying both $A$ and $S$
(resp. $A_1$ and $S_1$) on the right (resp. left) by a non-zero element of $\mc K[\partial]$
(Proposition \ref{20111003:thm2}(b)).
Though this result is very simple and natural, we were not able to find it in the literature.

In early 50's Leray \cite{Ler53} introduced an important generalization of the characteristic matrix
of a matrix pseudodifferential operator $A$.
Using this generalization Hufford in \cite{Huf65} developed a method to compute 
the Dieudonn\`e determinant $\det_1 A$
(see Section \ref{sec:4} for its definition).
Based on this method Sato and Kashiwara \cite{SK75} 
and Miyake \cite{Miy83} proved that $\det_1 A$ is holomorphic provided that $A$ has holomorphic coefficients.
In the present paper we extend this result and its proof from \cite{Miy83}
to the case when the algebra of holomorphic functions is replaced by an arbitrary 
differential domain $\mc A$
(Theorem \ref{20111208:thm2}).
Namely we show that if the coefficients of all entries of the matrix $A$ lie in the domain $\mc A$,
then its determinant $\det_1 A$ lies in the integral closure of $\mc A$
in its differential field of fractions $\mc K$.

A simple example when $\det_1 A$ does not lie in $\mc A$ itself is the following:
$$
A=\left(\begin{array}{ll}
a\partial & b\partial+\beta \\
b\partial-\beta & d\partial
\end{array}\right)\,,
$$
where $\mc A=\mb C[a^{(n)},b^{(n)},d^{(n)}\,|\,n\in\mb Z_+]/\{ad-b^2\}$,
$\partial x^{(n)}=x^{(n+1)}$ for $x=a,b,c$,
$\beta\in\mb C\backslash\{0\}$,
and $\{p\}$ denotes the differential ideal generated by $p\in\mc A$.
Then $\det_1 A=\beta^2-\beta a\left(\frac ba\right)'$
does not lie in the domain $\mc A$.

We derive from the proof of Theorem \ref{20111208:thm2} that,
for an integrally closed differential domain $\mc A$, a matrix pseudodifferential operator 
$A$ is invertible in the ring $\Mat_{\ell\times\ell}\mc A((\partial^{-1}))$
if and only if $\det_1A$ is an invertible element of $\mc A$
(Theorem \ref{20111215:thm1}).

The above problems arose in the paper \cite{DSK12} on non-local Hamiltonian structures.

We wish to thank Pavel Etingof and Andrea Maffei for useful discussions, and Toby Stafford for correspondence.


\section{Differential and pseudodifferential operators over a differential field}
\label{sec:2}

Let $\mc K$ be a differential field of characteristic zero, with the derivation $\partial$.
For $a\in\mc K$, we denote $a'=\partial(a)$ and $a^{(n)}=\partial^n(a)$, for a non negative integer $n$.
We denote by $\mc C\subset\mc K$ the subfield of \emph{constants},
i.e. $\mc C=\{\alpha\in\mc K\,|\,\alpha'=0\}$.

Recall that a \emph{pseudodifferential operator} over $\mc K$
is an expression of the form
\begin{equation}\label{20111003:eq1}
A(\partial)
=\sum_{n=-\infty}^N a_n \partial^n
\,\,,\,\,\,\, a_n\in\mc K,\,N\in\mb Z\,.
\end{equation}
If $a_N\neq0$, one says that $A(\partial)$ has \emph{order} $\ord(A)=N$,
and $a_N$ is called its \emph{leading coefficient}.
One also lets $\ord(A)=-\infty$ for $A=0$.
Pseudodifferential operators form a unital associative algebra, 
denoted by $\mc K((\partial^{-1}))$,
with product $\circ$ defined by letting
\begin{equation}\label{20111130:eq1}
\partial^n\circ a=\sum_{j\in\mb Z_+}\binom nj a^{(j)}\partial^{n-j}
\,\,,\,\,\,\, n\in\mb Z,a\in\mc K\,.
\end{equation}
We will often omit $\circ$ if no confusion may arise.
Obviously, for non-zero $A,B\in\mc K((\partial^{-1}))$
we have $\ord(AB)=\ord(A)+\ord(B)$, and the leading coefficient of $AB$ is the product of their leading coefficients.

The algebra $\mc K((\partial^{-1}))$ is a skewfield extension of $\mc K$.
Indeed, if $A(\partial)\in\mc K((\partial^{-1}))$ is a non-zero pseudodifferential operator
of order $N$ as in \eqref{20111003:eq1},
its inverse $A^{-1}(\partial)\in\mc K((\partial^{-1}))$ is computed as follows.
We write
$$
A(\partial)
=a_N\Big(1+\sum_{n=-\infty}^{-1} a_N^{-1}a_{n+N} \partial^n\Big)\partial^N\,,
$$
and expanding by geometric progression, we get
\begin{equation}\label{20111130:eq2}
A^{-1}(\partial)
=
\partial^{-N}\circ \sum_{k=0}^\infty\Big(-\sum_{n=-\infty}^{-1} a_N^{-1}a_{n+N} \partial^n\Big)^k\circ a_N^{-1}\,,
\end{equation}
which is well defined as a pseudodifferential operator in $\mc K((\partial^{-1}))$,
since, by formula \eqref{20111130:eq1},
the powers of $\partial$ are bounded above by $-N$,
and the coefficient of each power of $\partial$ is a finite sum.

The \emph{symbol} of the pseudodifferential operator $A(\partial)$ in \eqref{20111003:eq1}
is the formal Laurent series
$A(\lambda)=\sum_{n=-\infty}^N a_n \lambda^n\,\in\mc K((\lambda^{-1}))$,
where $\lambda$ is an indeterminate commuting with $\mc K$.
We thus get a bijective map $\mc K((\partial^{-1}))\to\mc K((\lambda^{-1}))$
(which is not an algebra homomorphism).
A closed formula for the associative product in $\mc K((\partial^{-1}))$
in terms of the corresponding symbols is the following:
\begin{equation}\label{20111003:eq2}
(A\circ B)(\lambda)=A(\lambda+\partial)B(\lambda)\,.
\end{equation}
Here and further on, we always expand an expression as $(\lambda+\partial)^{n},\,n\in\mb Z$, 
in non-negative powers of $\partial$:
\begin{equation}\label{20111004:eq1}
(\lambda+\partial)^{n}=\sum_{j=0}^\infty\binom nj \lambda^{n-j}\partial^j\,.
\end{equation}
Therefore, the RHS of \eqref{20111003:eq2} means
$\sum_{m,n=-\infty}^N\sum_{j=0}^\infty \binom{m}{j}a_m b_n^{(j)} \lambda^{m+n-j}$.

The algebra (over $\mc C$) $\mc K[\partial]$ of \emph{differential operators} over $\mc K$
consists of operators of the form 
$$
A(\partial)=\sum_{n=0}^Na_n\partial^n\,,\,\,a_n\in\mc K,\,N\in\mb Z_+\,.
$$
It is a subalgebra of $\mc K((\partial^{-1}))$, and a bimodule over $\mc K$.

\begin{proposition}
The algebra $\mc K[\partial]$ of differential operators over the differential field $\mc K$ is 
right (respectively left) \emph{Euclidean},
i.e. for every $A,B\in\mc K[\partial]$, with $A\neq0$,
there exist unique $Q,R\in\mc K[\partial]$ such that $B=AQ+R$ (resp. $B=QA+R$)
and either $R=0$ or $\ord(R)<\ord(A)$.
\end{proposition}
\begin{proof}
First, we prove existence of $Q$ and $R$.
If $B=0$, then we can take $Q=R=0$.
If $B\neq0$ and $\ord(B)<\ord(A)$, we can take $Q=0$ and $R=B$.
For $B\neq0$ and $\ord(B)\geq\ord(A)$, we proceed by induction on $\ord(B)$.
If $\ord(B)=0(=\ord(A))$, then $A,B$ lie in $\mc K$, and we can take $Q=\frac BA,\, R=0$.
Finally, consider the case when $\ord(A)=m\geq0$ and $\ord(B)=n\geq1$, with $n\geq m$.
Letting $a_m$ be the leading coefficeint of $A$ and $b_n$ be the leading coefficient of $B$,
the differential operator $\tilde B=B-A\frac{b_n}{a_m}\partial^{n-m}$
has order $\ord(\tilde B)<n$.
Hence, we can apply the inductive assumption to find $\tilde Q$ and $\tilde R$ in $\mc K[\partial]$
such that
$\tilde B=A\tilde Q+\tilde R$ and $\ord(\tilde R)<m$.
Then, letting $Q=\frac{b_n}{a_m}\partial^{n-m}+\tilde Q$ and $R=\tilde R$,
we get the desired decomposition.
As for the uniqueness, if one has $AQ+R=AS+T$ with $\max(deg(R),deg(T))<deg(A)$ then $A(Q-S)=T-R$. 
Comparing the orders of both sides, we conclude that $Q=S$, hence $R=T$.
The proof of the left Euclidean condition is the same.
\end{proof}
By the standard argument, we obtain the following:
\begin{corollary}\label{20111205:cor}
Every non-zero right (resp. left) ideal $\mc I$ of $\mc K[\partial]$ is principal:
$\mc I=A\mc K[\partial]$ (resp. $\mc I=\mc K[\partial]A$),
and it is generated by its element $A\in\mc I$ of minimal order 
(defined up to multiplication on the right (resp. left) by a non-zero element of $\mc K$).
\end{corollary}
\begin{remark}\label{20111202:rem}
If $\mc A$ is a differentail domain, with field of fractions $\mc K$,
then $\mc K((\partial^{-1}))$ is the skewfield of fractions of $\mc A((\partial^{-1}))$.
Indeed, if $A\in\mc A((\partial^{-1}))$ is non-zero,
then equation \eqref{20111130:eq2} provides its inverse in $\mc K((\partial^{-1}))$.
\end{remark}
\begin{remark}\label{20111003:rem}
If $\mc A$ is a (commutative) differential ring,
and $A\in\mc A((\partial^{-1}))$ is an element whose leading coefficient is not a zero divisor,
then equation \eqref{20111130:eq2} still makes sense, showing 
that $A$ is invertible in $\mc Q\mc A((\partial^{-1}))$,
where $\mc Q\mc A$ is the ring of fractions of $\mc A$, 
obtained by inverting all non zero divisors of $\mc A$.
In general, though,
the leading coefficient of an invertible pseudodifferential operator
$A\in\mc A((\partial^{-1}))$ is either invertible or a zero divisor.
As an example of the latter case,
let $\mc A=\mb F\oplus\mb F$, with $\partial$ acting as zero.
Then
$A(\partial)=(0,1)+(1,0)\partial$ is invertible
and its inverse is $(0,1)+(1,0)\partial^{-1}$.
\end{remark}


\section{Rational pseudodifferential operators}
\label{sec:3}

As before, let $\mc K$ be a differential field with derivation $\partial$.
\begin{definition}\label{20110926:def}
The algebra $\mc K(\partial)$ of \emph{rational pseudodifferential operators} over $\mc K$
is the smallest subskewfield of $\mc K((\partial^{-1}))$ containing $\mc K[\partial]$.
\end{definition}
We can describe the algebra $\mc K(\partial)$ of rational pseudodifferential operators 
more explicitly using the Ore theory (for a review of the Ore theory, 
in the case of an arbitrary non-commutative domain, see for example \cite{Art99}).
\begin{lemma}\label{20111205:lem}
The algebra $\mc K[\partial]$ satisfies the right (resp. left) Ore condition:
for every $A,B\in\mc K[\partial]$,
there exist $A_1,B_1\in\mc K[\partial]$ not both zero such that
\begin{equation}\label{20111003:eq3}
AB_1=BA_1
\,\,\,\,
\Big(\text{resp. } B_1A=A_1B\Big)
\,.
\end{equation}
\end{lemma}
\begin{proof}
If $A=0$, the statement is obvious since we can take $A_1=0$.
If $A\neq0$, we prove the claim by induction on $\ord(A)$.
Since $\mc K[\partial]$ is Euclidean,
there exist $Q,R\in\mc K[\partial]$ such that $B=AQ+R$ 
and either $R=0$ or $\ord(R)<\ord(A)$.
If $R=0$, then we can take $B_1=Q$ and $A_1=1$ and we are done.
If $R\neq0$, then by inductive assumption
there exist $A_1,R_1\in\mc K[\partial]$, both non-zero, such that $RA_1=AR_1$.
Hence, $BA_1=A(QA_1+R_1)$, therefore the claim holds with $B_1=QA_1+R_1$.
The left Ore condition is proved in the same way.
\end{proof}
\begin{remark}\label{20111206:rem1}
By Lemma \ref{20111205:lem}, if $A,B$ are non-zero elements of $\mc K[\partial]$,
then the right (resp. left) principal ideals generated by them have a non-zero intersection.
By Corollary \ref{20111205:cor} the intersection is again a right (resp. left) principal ideal,
generated by some element $M$, defined uniquely up to multiplication by a non-zero element of $\mc K$ 
on the right (resp. left).
This element $M\in\mc K[\partial]$ is the right (resp. left) 
\emph{least common multiple} (lcm) of $A$ and $B$.

Also, by Corollary \ref{20111205:cor} the sum of the right (resp. left) ideals generated by $A$ and $B$
is a right (resp. left) principal ideal, generated by some element $D$, 
defined uniquely up to multiplication by a non-zero element of $\mc K$ 
on the right (resp. left).
This element $D\in\mc K[\partial]$ is the right (resp. left) 
\emph{greatest common divisor} (gcd) of $A$ and $B$.
\end{remark}
\begin{proposition}\label{20111003:thm2}
\begin{enumerate}[(a)]
\item
The skewfield of rational pseudodifferential operators over $\mc K$ is
$$
\mc K(\partial)
=\big\{AS^{-1}\,\big|\, A,S\in\mc K[\partial],\,S\neq0\big\}
=\big\{S^{-1}A\,\big|\, A,S\in\mc K[\partial],\,S\neq0\big\}\,.
$$
In other words,
every rational pseudodifferential operator $L\in\mc K(\partial)$
can be written as a right and a left fraction
$L=AS^{-1}=S_1^{-1}A_1$
for some $A,A_1,S,S_1\in\mc A[\partial]$ with $S,S_1\neq0$.
\item
The decompositions $L=AS^{-1}=S_1^{-1}A_1$ of an element $L\in\mc K(\partial)$
are unique if we require that $S$ and $S_1$ have minimal possible order and leading coefficient 1.
Any other decomposition $L=\tilde A\tilde S^{-1}$
(resp. $L=\tilde S_1^{-1}\tilde A_1$),
with $\tilde A,\tilde A_1,\tilde S\tilde S_1\in\mc K[\partial]$, $\tilde S,\tilde S_1\neq0$,
can be obtained from the minimal one by multiplying both $A$ and $S$
(resp. $A_1$ and $S_1$) by the same non-zero factor from $\mc K[\partial]$ on the right (resp. left).
\end{enumerate}
\end{proposition}
\begin{proof}
The proof of part (a) can be deduced from the general Ore Theorem.
Since in this case we have explicit realizations of the sets
of right and of left fractions of $\mc A[\partial]$ as subsets of $\mc K((\partial^{-1}))$,
we can give a direct proof using the right and left Ore conditions.

In order to prove part (a) it sufficies to show
that the sets
$\mc S_{\text{right}}=\big\{AS^{-1}\,\big|\, A,S\in\mc K[\partial],\,S\neq0\big\}$
and $\mc S_{\text{left}}=\big\{S^{-1}A\,\big|\, A,S\in\mc K[\partial],\,S\neq0\big\}$
are closed under addition and multiplication.
Let $A,B,S,T\in\mc K[\partial]$,
with $S,T\neq0$.
By the right Ore condition \eqref{20111003:eq3}, there exist non-zero
$S_1,\,T_1\in\mc F[\partial]$ 
such that $ST_1=TS_1$.
Hence,
$$
AS^{-1}+BT^{-1}
=
\big(AT_1+BS_1\big)\big(S\circ T_1\big)^{-1}\,,
$$
proving that $\mc S_{\text{right}}$ is closed under addition.
Again by the right Ore condition, there exist
$S_1,\,B_1\in\mc A[\partial]$, 
such that $S_1\neq0$
and $SB_1=BS_1$.
Hence,
$$
AS^{-1}\circ BT^{-1}
=
\big(A\circ B_1\big)\big(T\circ S_1\big)^{-1}\,,
$$
proving that $\mc S_{\text{right}}$ is closed under multiplication.
Similarly for the set $\mc S_{\text{left}}$ of left fractions.

For part (b), consider the set
$$
\mc I=\big\{S\in\mc K[\partial]\backslash\{0\}\,\big|\,L=AS^{-1} \text{ for some } A\in\mc K[\partial]\big\}\cup\{0\}
\,.
$$
We claim that $\mc I$ is a right ideal of $\mc K[\partial]$.
First, if $L=AS^{-1}$ and $0\neq T\in\mc K[\partial]$, then $L=AT(ST)^{-1}$,
proving that $\mc IT\subset\mc I$.
Moreover, if $L=A_1S_1^{-1}=A_2S_2^{-1}$, with $S_1+S_2\neq0$,
then $L=(A_1+A_2)(S_1+S_2)^{-1}$, proving that $\mc I$ is closed under addition.
By Corollary \ref{20111205:cor}
there is a unique monic element $S$ in $\mc I$ of minimal order,
and every other element of $\mc I$ is obtained from $S$ by a multiplication by a non-zero element of $\mc K[\partial]$.
This proves part (b) for right fractions. The proof for left fractions is the same.
\end{proof}
\begin{remark}\label{20111205:rem2a}
If $\mc A$ is a differential domain and $\mc K$ is its field of fractions,
then we define $\mc A(\partial)=\mc K(\partial)$, 
and it is easy to see, clearing the denominators, that 
all its elements are of the form $AS^{-1}$ (or $S^{-1}A$)
for $A,S\in\mc A[\partial],\,S\neq0$.
\end{remark}
\begin{remark}\label{20111205:rem2b}
If $\mc A$ is a differential domain,
we can ask whether the right (resp. left) Ore condition holds
for any multiplicative subset $\mc S\subset\mc A[\partial]$:
for every $A\in\mc A[\partial]$ and $S\in\mc S$, there exist $A_1\in\mc A[\partial]$
and $S_1\in\mc S$ such that $AS_1=SA_1$ (resp. $S_1A=A_1S$).
In fact, this is false,
as the following example shows.
Consider the algebra of differential polynomials in one variable,
$\mc A=\mb C[u^{(n)},\,n\in\mb Z_+]$, where $\partial(u^{(n)})=u^{(n+1)}$,
and let $\mc S\subset\mc A[\partial]$ be the multiplicative subset consisting of
differential operators $A\in\mc A[\partial]$ with leading coefficient 1.
Letting $A=u\in\mc A[\partial]$ and $S=\partial\in\mc S$,
we find $A_1=u^2,\,S_1=u\partial+2u'\in\mc A[\partial]$
such that $AS_1=SA_1$, but it is not hard to prove that $S_1$ cannot be chosen with leading coefficient 1
(unless we allow to have the other coefficients in the field of fractions $\mc K$).
This example provides an element $\partial^{-1}\circ u=u^2(u\partial+2u')^{-1}\in\mc A((\partial^{-1}))$
which is a left fraction but not a right fraction
(i.e. it is not of the form $AS^{-1}$ with $A\in\mc A[\partial]$ and $S\in\mc S$).
\end{remark}
\begin{remark}\label{20111205:rem2c}
If $\mc A$ is a differential (commutative associative) ring possibly with zero divisors,
we can define $\mc A(\partial)$ as the ring generated by $\mc A[\partial]$
and the inverses of all elements $S\in\mc A[\partial]$ which are invertible in $\mc A((\partial^{-1}))$.
This contains both sets $\mc S_{\text{right}}$ and $\mc S_{\text{left}}$ of right and left fractions,
$$
\begin{array}{l}
\mc S_{\text{right}}=\big\{AS^{-1}\,\big|\,A\in\mc A[\partial],S\in\mc A[\partial]\cap\mc A((\partial^{-1}))^\times\big\}
\,,\\
\mc S_{\text{left}}=\big\{S^{-1}A\,\big|\,A,S\in\mc A[\partial],S\in\mc A[\partial]\cap\mc A((\partial^{-1}))^\times\big\}
\,,
\end{array}
$$
but, in general, these two sets are not equal.
An example when $\mc S_{\text{left}}$ and $\mc S_{\text{right}}$ are not equal
was provided in Remark \ref{20111205:rem2b}:
$\partial^{-1}\circ u$ lies in $\mc S_{\text{left}}$ but not in $\mc S_{\text{right}}$.
Note though that, when $\mc A$ is a domain, this definition of $\mc A(\partial)$ 
is NOT the same as $\mc K(\partial)$ 
(which was the definition of $\mc A(\partial)$ given in Remark \ref{20111205:rem2a} in the case of domains).
\end{remark}
\pecettaold{
When $\mc A$ is a domain, we have the sets:
$$
\begin{array}{l}
\mc S_{\text{right}}\,,\,\,\mc S_{\text{left}}
\subset
\mc A(\partial)=
\Big\langle AS^{-1}\,\Big|\,A\in\mc A[\partial],S\in\mc A[\partial]\cap\mc A((\partial^{-1}))^\times \Big\rangle
\\
\subset
\mc K(\partial)\cap\mc A((\partial^{-1}))=
\Big\{ AS^{-1}\,\Big|\,A,S\in\mc A[\partial],S\neq0,AS^{-1}\in\mc A((\partial^{-1})) \Big\}
\end{array}
$$
The first inlusion is strict (see example in remark).
QUESTION: is the second inclusion equal?
}


\section{Matrix pseudodifferential operators}
\label{sec:4}

As in the previous sections, let $\mc K$ be a differential field with derivation $\partial$.
We recall here some linear algebra over the skewfield $\mc K((\partial^{-1}))$
and, in particular, the notion of the Dieudonn\'e determinant
(see \cite{Art57} for an overview over an arbitrary skewfield).

Let $\mc D$ be a subskewfield of $\mc K((\partial^{-1}))$.
We are interested in the case when $\mc D=\mc K(\partial)$ or $\mc K((\partial^{-1}))$.
An \emph{elementary row operation} of the matrix pseudodifferential operator 
$A\in\Mat_{m\times\ell}\mc D$
is either a permutation of two rows of it, 
or the operation $\mc T(i,j;P)$, where $1\leq i\neq j\leq m$ and $P\in\mc D$,
which replaces the $j$-th row by itself minus $i$-th row
multiplied on the left by $P$.
Using the usual Gauss elimination, we can get the analogues
of standard linear algebra theorems for matrix pseudodifferential operators.
In particular, 
any $m\times\ell$ matrix pseudodifferential operator $A$ with entries in $\mc D$
can be brought by elementary row operations to a row echelon form (over $\mc D$).

The \emph{Dieudonn\'e determinant} of $A\in\Mat_{\ell\times\ell}\mc K((\partial^{-1}))$
has the form $\det A=c\lambda^d$, where $c\in\mc K$, $\lambda$ is an indeterminate, and $d\in\mb Z$.
It is defined by the following properties:
$\det A$ changes sign if we permute two rows of $A$,
and it is unchanged under any elementary row operation $\mc T(i,j;P)$ defined above,
for aribtrary $i\neq j$ and a pseudodifferential operator $P\in\mc K((\partial^{-1}))$;
furthermore, if $A$ is upper triangular,
with non-zero diagonal entries $A_{ii}\in\mc K((\partial^{-1}))$ of order $n_i$ and leading coefficient $a_i\in\mc K$,
then 
$$
\det A=(\det{}_1A)\lambda^{\d(A)}
\,\,,\,\,\,\,
\text{ where }\,\,
\det{}_1A=\prod_{i=1}^\ell a_i
\,\,\,\,\text{ and }\,\,
\d(A)=\sum_{i=1}^\ell n_i\,,
$$
and, if one of the diagonal entries is zero, we let $\det A=0$, $\det_1A=0$  and $\d(A)=-\infty$.
It follows from the results in \cite{Die43} that the Dieudonn\'e determinant is well defined
and $\det(AB)=(\det A)(\det B)$
for every $A,B\in\Mat_{\ell\times\ell}\mc K((\partial^{-1}))$.
\begin{remark}\label{20111212:rem1}
If $A\in\Mat_{\ell\times\ell}\mc K[\partial]$ is a matrix differential operator,
then it can be brought to an upper triangular form by elementary transformations
involving only differential operators (see \cite[Lemma A.2.6]{DSK11}).
Hence, if $\det A\neq0$, then $\d(A)$ is a non-negative integer.
\end{remark}
\begin{remark}\label{20111215:rem}
Let $A\in\Mat_{\ell\times\ell}\mc K((\partial^{-1}))$,
and denote by $A^*$ its adjoint matrix.
Then $\det A=(-1)^{\d(A)}\det A^*$.
Indeed, if $A=ET$, where $E$ is product of elementary matrices and $T$ is upper triangular,
then $A^*=T^*E^*$ and, clearly, $\det E^*=\det E$, while $\det T^*=(-1)^{\d(A)}\det T$.
However, in general $\det A$ and $\det A^T$ are not related to each other in an obvious way,
as the following example shows:
for 
$A=\left(\begin{array}{ll}
\partial & x\partial \\ 1 & x
\end{array}\right)$,
we have $\det A=1$, while $\det A^T=0$.
Note that, in this example, the matrix $A$ is invertible in $\Mat_{2\times2}\mc K((\partial^{-1}))$,
its inverse being
$A^{-1}=\left(\begin{array}{ll}
x & -x\partial+1 \\ -1 & \partial
\end{array}\right)$,
while $A^T$ is not invertible.
\end{remark}
Our main interest in the Deudonn\'e determinant is that it gives a way to characterize invertible
matrix pseudodifferential operators.
This is stated in the following:
\begin{proposition}\label{20111005:prop2}
\begin{enumerate}[(a)]
\item
An element $A\in\Mat_{\ell\times\ell}\mc K((\partial^{-1}))$ is invertible if and only if $\det A\neq0$.
\item
Suppose that $A\in\Mat_{\ell\times\ell}\mc K((\partial^{-1}))$ has $\det A\neq0$.
Then $A\in\Mat_{\ell\times\ell}(\mc D)$ if and only if $A^{-1}\in\Mat_{\ell\times\ell}(\mc D)$.
\end{enumerate}
\end{proposition}
\begin{proof}
As noted above, by performing elementary row operations, 
we can write $A=ET$, 
where $E$ is product of elementary matrices, and $T$ is an upper triangular matrix.
Then $A$ is invertible in the algebra
$\Mat_{\ell\times\ell}\mc K((\partial^{-1}))$ if and only if $T$ is invertible,
and this happens if and only if all the diagonal entries 
$T_{11},\dots,T_{\ell\ell}$ are non zero.
On the other hand, we have $\det A=0$ if one of the $T_i$'s is zero,
while, otherwise, denoting by $n_i$ the order of $T_{ii}$ and by $t_i$ its leading coefficient,
we have
$\det A=\pm t_1\dots t_\ell\lambda^{n_1+\dots+n_\ell}\neq0$.
This proves part (a).
For part (b) it sufficies to note that if $A$ has entries in $\mc D$, so do $E$ and $T$,
and, therefore, $A^{-1}$.
\end{proof}
\pecettaold{
It is not clear how to define \emph{rational matrices} with entries in $\mc A((\partial^{-1}))$,
where $\mc A$ is an arbitrary differential ring, or even a domain.
Posibilities:
\begin{enumerate}[(i)]
\item 
if $\mc A$ is a domain, with field of fractions $\mc K$, we can take $\Mat_{\ell\times\ell}\mc K(\partial)$,
\item
we could take the smallets subring of $\Mat_{\ell\times\ell}\mc A((\partial^{-1}))$
generated by $\Mat_{\ell\times\ell}\mc A[\partial]$
and the inverse of matrices $S\in\Mat_{\ell\times\ell}\mc A[\partial]$
which are invertible in $\Mat_{\ell\times\ell}\mc A((\partial^{-1}))$.
\end{enumerate}
It is not clear how these definitions are related 
(basically because we don't have a Dieudonn\`e determinant, and we don't have Ore condition).

CONJECTURE: for a domain the above two notions are related by:
$$
\text{def.}(ii)=\text{def.}(ii)\cap\Mat{}_{\ell\times\ell}\mc A((\partial^{-1}))\,.
$$
}
\begin{remark}\label{20111007:rem1}
Suppose $\mc A$ is a differential domain with field of fractions $\mc K$.
Let $A\in\Mat_{\ell\times\ell} \mc A[\partial]$.
By Proposition \ref{20111005:prop2}, if $\det A\neq0$,
then $A^{-1}$ exists and it has entries in $\Mat_{\ell\times\ell} \mc K((\partial^{-1}))$.
In fact, 
it is clear from the proof that its entries lie in $\tilde{\mc A}((\partial^{-1}))$,
where $\tilde{\mc A}$ is an extension of $\mc A$ obtained by adding inverses of finitely many non-zero elements.
\end{remark}
\begin{definition}\label{20111208:def1}
Let $A=\big(A_{ij}\big)_{i,j=1}^\ell\in\Mat_{\ell\times\ell}\mc K((\partial^{-1}))$.
The \emph{total order} of $A$ is defined as
$$
\tord(A)=\max_{\sigma\in S_\ell} \Big(\sum_{i=1}^{\ell} \ord(A_{i,\sigma(i)})\Big)\,,
$$
where $S_\ell$ denotes the group of permutations of $\{1,\dots,\ell\}$.
Assuming $\det A\neq0$, we define the \emph{degeneracy degree} of $A$ as 
$$
\dd(A)=\tord(A)-\d(A)\,,
$$
(by Theorem \ref{20111208:thm1}(i) below, $dd(A)$ is a non-negative integer),
and we say that $A$ is \emph{strongly non-degenerate} if $\dd(A)=0$.
\end{definition}
%
%
\begin{definition}\label{20111207:def1}
Let $A=\big(A_{ij}\big)_{i,j=1}^\ell\in\Mat_{\ell\times\ell}\mc K((\partial^{-1}))$.
A system of integers $(N_{1},...,N_{\ell},h_{1},...,h_{\ell})$ 
is called a \emph{majorant} of $A$
if 
$$
\ord(A_{ij})\leq N_j - h_i
\,\,\text{ for every } i,j=1,\dots,\ell\,.
$$ 
The \emph{characteristic matrix}
$\bar A(\lambda)=\big(\bar A_{ij}(\lambda)\big)_{i,j=1}^\ell\in\Mat_{\ell\times\ell}\mc K[\lambda^{\pm1}]$,
associated to the majorant $\{N_i,h_i\}_{i=1}^\ell$,
is defined by 
$$
\bar A_{ij}(\lambda)=a_{ij;N_j-h_i}\lambda^{N_j-h_i}\,,
$$
where $a_{ij;N_j-h_i}$ is the coefficient of $\partial^{N_j-h_i}$ in $A_{ij}$.
Clearly, 
\begin{equation}\label{20111209:eq1}
\sum_{i=1}^{n} ({N_{i}}-{h_{i}})\geq\tord(A)
\end{equation}
for every majorant $\{N_i,h_i\}_{i=1}^\ell$ of $A$.
A majorant is called \emph{optimal} if 
$$
\sum_{i=1}^{n} ({N_{i}}-{h_{i}})=\tord(A)\,.
$$
\end{definition}
The following theorem follows from the results in \cite{Huf65}.
\begin{theorem}\label{20111208:thm1}
Let  $A=\big(A_{ij}\big)_{i,j=1}^\ell\in\Mat_{\ell\times\ell}\mc K((\partial^{-1}))$
with $\det A\neq0$.
We have:
\begin{enumerate}[(i)]
\item
$\dd(A)\geq 0$;
\item
if $\det A\neq0$, then there exists an optimal majorant of $A$;
\item
if $\dd(A)\geq1$, then $\det(\bar A(\lambda))=0$ for any majorant;
\item
if $\dd(A)=0$, then $\det(\bar A(\lambda))=0$ for any majorant which is not optimal,
and $\det(\bar A(\lambda))=\det A$ for any majorant which is optimal.
\end{enumerate}
\end{theorem}
In the special case when 
$A$ is an $\ell\times\ell$ matrix pseudodifferential operator of order $N$
with invertible leading coefficient $A_N\in\Mat_{\ell\times\ell} \mc K$,
we can take the (optimal) majorant $N_i=N,\,h_i=0,\,i=1,\dots,\ell$.
The corresponding characteristic matrix is $\bar A(\lambda)=A_N\lambda^N$.
We thus obtain the following
\begin{corollary}\label{20111005:prop1}
If $A$ is an $\ell\times\ell$ matrix pseudodifferential operator of order $N$
with invertible leading coefficient $A_N\in\Mat_{\ell\times\ell} \mc K$, then
$\det A=(\det A_N) \lambda^{N\ell}$.
\end{corollary}

The proof of the following result is similar to that in 
in \cite{SK75} and \cite{Miy83} in the case when $\mc A$ is the algebra of holomorphic functions
in a domain of the complex plain.
\begin{theorem}\label{20111208:thm2}
Let $\mc A$ be a unital differential subring of the differential field $\mc K$,
and let $\bar{\mc A}$ be its integral closure in $\mc K$.
Then, for any $A\in\Mat_{\ell\times\ell}\mc A((\partial^{-1}))$ we have $\det_1A\in\bar{\mc A}$.
\end{theorem}
Let $\mc B$ be a valuation ring in $\mc K$ containing $\mc A$.
By Proposition \ref{20111210:prop1}, it sufficies to prove that $\det_1A\in\mc B$.
This follows from the following two lemmas.
\begin{lemma}\label{20111208:lem1}
Suppose that $A\in\Mat_{\ell\times\ell} \mc B((\partial^{-1}))$
has $\det A\neq0$ and degeneracy degree $\dd(A)\geq1$.
Then there exists a matrix 
$P \in \Mat_{\ell\times\ell} \mc B((\partial^{-1}))$
such that $\det_1P=1$ and $\dd(PA)\leq\dd(A)-1$.
\end{lemma}
\begin{proof}
By Theorem \ref{20111208:thm1}(ii), there exists an optimal majorant $\{N_i,h_i\}_{i=1}^\ell$ for $A$.
%
Since, by assumption, $\dd(A)\geq1$,
by Theorem \ref{20111208:thm1}(iii), the characteristic matrix $\bar A(\lambda)$ associated to
this optimal majorant is degenerate
(and so is $\bar A(1)\in\Mat_{\ell\times\ell}\mc K$).
Let $(f_1,\dots,f_\ell)\in\mc K^\ell$ be a left eigenvector of $\bar A(1)$ 
with eigenvalue $0$.
Since $\mc B$ is a valuation ring, condition $(A_\ell)$ in Proposition \ref{20111210:prop2} 
of the Appendix holds.
Hence, after dividing all the entries $f_j$ by a non-zero entry $f_i$,
we may assume that $f_j\in\mc B$ for all $j=1,\dots,\ell$ and $f_i=1$ for some $i$.
Consider the following matrix 
\begin{equation}\label{20111214:eq1}
P=
\left(\begin{array}{lllll}
\partial^{h_1-h_\ell} & 0 & \dots & 0 & 0 \\
0 & \partial^{h_2-h_\ell} & \dots & 0 & 0 \\
& \vdots &  & \vdots & \\
f_1\partial^{h_1-h_\ell} & f_2\partial^{h_2-h_\ell} & \dots & f_{\ell-1}\partial^{h_{\ell-1}-h_\ell} & f_\ell \\
& \vdots &  & \vdots & \\
0 & 0 & \dots & \partial^{h_{\ell-1}-h_\ell} & 0 \\
0 & 0 & \dots & 0 & 1 \\
\end{array}\right)
\begin{array}{lllll}
\\ \\ \text{row } i \\ \\ \\
\end{array}
\end{equation}
Note that $P\in\Mat_{\ell\times\ell} \mc B((\partial^{-1}))$,
it is stronlgy non-degenerate 
(since the characteristic matrix associated to the majorant $\{h_j-h_\ell,0\}_{j=1}^\ell$
is non-degenerate) and its Dieudonn\`e determinant is 
$$
\det P=\lambda^{\sum_j(h_j-h_\ell)}\,,
$$
i.e. $\det_1P=1$ and $\d(P)=\sum_{j=1}^\ell(h_j-h_\ell)$.
We claim that the following is a majorant for the matrix $PA$:
\begin{equation}\label{20111209:eq2}
\{N_j,h_\ell+\delta_{j,i}\}_{j=1}^\ell\,.
\end{equation}
Indeed, for $j\neq i$ we have
$$
\ord(PA)_{jk}
=\ord(\partial^{h_j-h_\ell}A_{jk})
=h_j-h_\ell+\ord(A_{jk})
\leq 
N_k-h_\ell\,.
$$
For the entries in the $i$-th row of the matrix $PA$ we have
$$
(PA)_{ik}=\sum_{j=1}^\ell f_j\partial^{h_j-h_\ell} A_{jk}\,,
$$
which has order less then or equal to $\max_j\{h_j-h_\ell+\ord(A_{jk})\}\leq N_k-h_\ell$,
and the coefficient of $\partial^{N_k-h_\ell}$ in $(PA)_{ik}$ is
$$
\sum_{j=1}^\ell f_j a_{jk;N_k-h_j}
=\sum_{j=1}^\ell f_j \bar A(1)_{jk}=0
\,.
$$
Hence, 
$$
\ord(PA)_{ik}\leq N_k-(h_\ell+1)\,,
$$
as we wanted.
By \eqref{20111209:eq1} (applied to the majorant \eqref{20111209:eq2} of $PA$), we have
$$
\begin{array}{l}
\displaystyle{
\tord(PA)\leq \sum_{j=1}^\ell (N_j-h_\ell-\delta_{j,i})
} \\
\displaystyle{
=\sum_{j=1}^\ell(N_j-h_j)+\sum_{j=1}^\ell(h_j-h_\ell)-1
=\tord(A)+\d(P)-1\,.
}
\end{array}
$$
The claim follows.
\end{proof}
\begin{lemma}\label{20111208:lem2}
If $A\in\Mat_{\ell\times\ell} \mc B((\partial^{-1}))$
has non-zero Dieudonn\`e determinant,
then there exists a matrix 
$P \in \Mat_{\ell\times\ell} \mc B((\partial^{-1}))$
such that $\det_1P=1$ and $PA$ is strongly non-degenerate.
\end{lemma}
\begin{proof}
Recall that, by Theorem \ref{20111208:thm1}(i), $\dd(A)\geq0$.
The claim follows from Lemma \ref{20111208:lem1} by induction on $\dd(A)$.
%
\end{proof}
\begin{proof}[Proof of Theorem \ref{20111208:thm2}]
By Lemma \ref{20111208:lem2} there exists 
$P \in \Mat_{\ell\times\ell} \mc B((\partial^{-1}))$
such that $\det_1P=1$ and $PA$ is strongly non-degenerate.
Therefore, if we fix an optimal majorant for the matrix $PA$
and let $\overline{PA}(\lambda)$ be the corresponding characteristic matrix,
we have
$\det(PA)=\det(\overline{PA}(\lambda))$.
On the one hand, we have that
$\det_1(PA)=(\det_1 P)(\det_1 A)=\det_1 A$.
On the other hand, $\overline{PA}(\lambda)$ has entries in $\mc B((\lambda^{-1}))$ (since both $P$ and $A$
have entries in $\mc B((\partial^{-1}))$),
therefore its determinant lies in $\mc B((\lambda^{-1}))$,
implying that $\det_1A\in\mc B$.
\end{proof}

\pecettaold{
QUESTION (for Sylvain):
there should be a more direct proof of the fact that $\det A$
lies in the integral closure of $\mc A$,
i.e. it is solution of some monic polynomial equation with coefficients in $\mc A$.
}

\begin{example}\label{20120104:ex1}
Consider an arbitrary $2\times2$ matrix differential operator of order 1 
over a differential domain $\mc A$:
$$
A=\left(\begin{array}{ll}
a\partial+\alpha & b\partial+\beta \\
c\partial+\gamma & d\partial+\delta
\end{array}\right)\,,
$$
where $a,b,c,d,\alpha,\beta,\gamma,\delta\in\mc A$.
We may assume, without loss of generality, that $a\neq0$.
We denote,
$$
\Delta_\lambda
=
\left|\begin{array}{ll}
a\lambda+\alpha & b\lambda+\beta \\
c\lambda+\gamma & d\lambda+\delta
\end{array}\right|\,.
$$
It can be expanded as
$\Delta_\lambda=\Delta_\infty\lambda^2+\Delta'_0\lambda+\Delta_0$, where
$$
\Delta_\infty
=
\left|\begin{array}{ll}
a & b \\ c & d
\end{array}\right|
\,\,,\,\,\,\,
\Delta'_0
=
\left|\begin{array}{ll}
a & \beta \\ c & \delta
\end{array}\right|
+
\left|\begin{array}{ll}
\alpha & b \\ \gamma & d
\end{array}\right|
\,\,,\,\,\,\,
\Delta_0
=
\left|\begin{array}{ll}
\alpha & \beta \\ \gamma & \delta
\end{array}\right|
\,.
$$
There are the following three possibilities for $\det A$:
\begin{enumerate}[1.]
\item
If $\Delta_\infty=ad-bc\neq0$, the matrix $A$ is strongly non-degenerate
of total order 2.
Its Dieudonn\`e determinant is
$\det A=\Delta_\infty\lambda^2$.
\item
If $\Delta_\infty=0$ and $\Delta'_0\neq0$, we have
$\det A=\Delta'_0\lambda$.
In this case, the matrix $A$ has total order 2 if $ad=bc\neq0$,
while it is strongly non-degenerate 
(of total order 1) if $ad=bc=0$.
\item
Finally, if $\Delta_\infty=\Delta'_0=0$,
we have
$$
\det A=\det{}_1A=\Delta_0-(\alpha c-\gamma a)\big(\frac{b}{a}\big)'\,.
$$
For the total order of $A$ there are several possibilities:
if $ad=bc\neq0$ then $\tord(A)=2$;
if $ad=bc=0$ and $(\alpha d,\delta a,\beta c,\gamma b)\neq(0,0,0,0)$ then $\tord(A)=1$;
finally if $ad=bc=\alpha d=\delta a=\beta c=\gamma b=0$, and $(\alpha\delta,\beta\gamma)\neq(0,0)$,
then the matrix $A$ is strongly non-degenerate of total order 0, unless $\det A=0$.
In general $\det A$ does not lie in the domain $\mc A$ (see Example \ref{20120104:ex} below).
On the other hand, 
from Theorem \ref{20111208:thm2} we know that $\det A$ is solution 
of a monic polynomial equation with coefficients in $\mc A$. In fact, 
in this example we have
$$
\begin{array}{c}
\det A+\det A^T=2\Delta_0-(\alpha d'+\delta a'-\beta c'-\gamma b') \,,\\
(\det A)(\det A^T)=\Delta_0^2-\Delta_0(\alpha d'+\delta a'-\beta c'-\gamma b' )
\\ 
+(\beta\gamma-\alpha\delta)(b'c'-a'd')+\alpha\delta a'd'
\,\in\mc A\,.
\end{array}
$$
Hence,  $\det A$ is a root of the following quadratic polynomial 
with coefficients in $\mc A$:
$$
x^2-(\det A+\det A^T)x+(\det A)(\det A^T)\,.
$$
\end{enumerate}
\end{example}
\begin{example}\label{20120104:ex}
In Example \ref{20120104:ex1}, if $\dd(A)=\tord(A)-\d(A)$ is equal to 0 or 1,
then $\det_1 A$ lies in the domain $\mc A$.
On the other hand, if $\dd(A)=2$ (i.e. $\tord(A)=2$ and $\d(A)=0$),
this is not necessarily the case.
To see  this, consider the algebra $\bar{\mc A}=\mb C[a,b,d]/(ad-b^2)$,
and let $\mc A=\mb C[a^{(n)},b^{(n)},d^{(n)}\,|\,n\in\mb Z_+]/\{ad-b^2\}$,
where $\{p\}$ denotes the differential ideal generated by $p$.
Note that $\bar{\mc A}$ is a domain,
hence, by Kolchin's Theorem \cite[Prop.IV.10]{Kol73}, $\mc A$ is a domain too.
Consider the matrix $A$ as in Example \ref{20120104:ex1}, with $c=b$ 
and $\alpha=\delta=0, \beta=-\gamma\in\mb C\backslash\{0\}$.
In this case we have $\Delta_\lambda=\Delta_0=\beta^2\in\mb C$,
and $\det A=\beta^2-\beta a\big(\frac{b}{a}\big)'$.
Note that in this example the determinant of the transposed matrix is not the same 
as the determinant of $A$, 
in fact we have $\det(A^T)=\beta^2+\beta a\big(\frac{b}{a}\big)'$.
We now show that, with these choices, $\det A$ does not lie in the ring $\mc A$.
Suppose, by contradiction, that $\det A\in\mc A$,
hence $\frac{ba'}{a}\in\mc A$.
Introduce the $\mb Z_+$ grading of $\mc A$ by letting $\deg(a^{(n)})=\deg(b^{(n)})=\deg(d^{(n)})=1$.
Then $\frac{ba'}{a}$ has degree 1, which means that  we have an equality of the form
$$
\frac{ba'}{a}=\sum_{n=0}^N \big( \alpha_na^{(n)}+\beta_nb^{(n)}+\delta_nd^{(n)} \big) \,,
$$
with $N\in\mb Z_+$ and $\alpha_n,\beta_n,\delta_n\in\mb C$.
We have an injective homomorphism of $\mb Z$ graded algebras
$\mc A\to\mb C[a^{-1},a^{(n)},b^{(n)}\,|\,n\in\mb Z_+]$
obtained by mapping $d^{(n)}\mapsto \big(\frac{b^2}{a}\big)^{(n)}$.
By assumption, $\frac{ba'}{a}\in\mb C[a^{-1},a^{(n)},b^{(n)}\,|\,n\in\mb Z_+]$
is in the image of this map,
and the above equality translates to the following 
equality in the ring $\mb C[a^{-1},a^{(n)},b^{(n)}\,|\,n\in\mb Z_+]$:
$$
\frac{ba'}{a}=\sum_{n=0}^N 
\Big( \alpha_na^{(n)}+\beta_nb^{(n)}+\delta_n\left(\frac{b^2}{a}\right)^{(n)} \Big) \,.
$$
Consider the quotient map
$\mb C[a^{-1},a^{(n)},b^{(n)}\,|\,n\in\mb Z_+]\to \mb C[a^{-1},a^{(n)},b\,|\,n\in\mb Z_+]$,
obtained by letting $b'=b''=b^{(3)}=\dots=0$.
The above equation translates to the following identity 
in the ring $\mb C[a^{-1},a^{(n)},b\,|\,n\in\mb Z_+]$:
$$
\frac{ba'}{a}=\sum_{n=0}^N \Big( \alpha_na^{(n)}+\beta_0b+b^2\delta_n\left(\frac{1}{a}\right)^{(n)} \Big) \,,
$$
which implies, looking at the coefficients of $b$ in both sides, that $\frac{a'}{a}=\beta_0\in\mb C$
in the ring $\mb C[a^{-1},a^{(n)}\,|\,n\in\mb Z_+]$, a contradiction.
\end{example}
\begin{conjecture}\label{20120105:conj}
If $\dd(A)=1$, then $\det_1 A$ lies in $\mc A$
(and we know that if $\dd(A)=0$ then $\det_1 A\in\mc A$).
\end{conjecture}
\pecettaold{
Conjecture: $\det A$ is expressed as polynomial in the coefficients of the entries of $A$
and their logarithmic derivatives (i.e. fractions $\frac{a'}{a}$, where $a$ is a coefficient in $A$).
}
\begin{remark}\label{20111214:rem}
The proof of Lemma \ref{20111208:lem1} provides an algorithm to compute $\det A$,
inductively on 
the degeneracy degree $\dd(A)=\tord(A)-\d(A)$.
Let $\{N_j,h_j\}_{j=1}^\ell$ be an optimal majorant for $A$.
If $\dd(A)=0$ (i.e. $A$ is strongly non-degenerate), then
the characteristic matrix $\bar A(\lambda)$ associated to this majorant is non-degenerate,
and in this case $\det A=\det \bar A(\lambda)=\det \bar A(1)\lambda^{\sum_j(N_j-h_j)}$.
If $\bar A(\lambda)$ is not strongly non-degenerate, we take a left null vector $(f_1,...,f_n)$ 
of the matrix $\bar A(1)$, normalized in such a way that ${f_i}=1$ for some $i$.
Then we multiply $A$ on the left by the matrix $P$ as in \eqref{20111214:eq1}. 
The resulting matrix $PA$ has determinant equal to $(\det A)\lambda^{\sum_j(h_j-h_\ell)}$,
and its degeneracy degree is at most $\dd(A)-1$.

The algorithm just described is based on an optimal majorant $\{N_j,h_j\}_{j=1}^\ell$.
We now give a constructive way of finding one.
To do so we consider the matrix of orders: $M=\big(m_{ij}\big)_{i,j=1}^\ell$,
where $m_{ij}=\ord(A_{ij})$.
Note that a majorant of $A$ only depends on $M$,
and we define the notion of total order and majorant of $M$ in the obvious way.
For a permutation $\sigma\in S_\ell$, we let
$$
m(\sigma)=\sum_{i=1}^\ell m_{i,\sigma(i)}\,.
$$
Hence, by definition, $\tord(M)=\max_{\sigma\in S_\ell}m(\sigma)$.
We then define the following subset of $\{1,\dots,\ell\}^2$:
$$
\Gamma(M)=\Big\{(i,j)\,\Big|\,\sigma(i)=j
\,\text{ for some } \sigma\in S_\ell \text{ such that } m(\sigma)=\tord(A)\Big\}\,.
$$
If $\Gamma(M)=\{1,\dots,\ell\}^2$ then $\{N_j=m_{1j},h_j=m_{11}-m_{j1}\}_{j=1}^\ell$
is an optimal majorant of $M$ \cite[Cor.III.2]{Huf65}.
For $i,j=1,\dots,\ell$, we let 
$$
\d{}_{ij}(M)=\tord(A)-\max_{\sigma\in S_\ell\,|\,\sigma(i)=j} m(\sigma)\,.
$$
Note that, by definition, $\Gamma(M)=\{(i,j)\,|\,\d_{ij}(M)=0\}$.
If $\d_{{i_1}{j_1}}(M)>0$ for some $(i_1,j_1)$, define $M_1$ as the matrix obtained from $M$ 
by replacing $m_{{i_1}{j_1}}$ with $(m_{{i_1}{j_1}} + d_{{i_1}{j_1}}(M))$. 
It is not hard to prove that $\tord(M_1)=\tord(M)$, 
and ${\sum_{i,j}{d_{ij}(M_1)}} < {\sum_{i,j}{d_{ij}(M)}}$.
By repeating this procedure $n$ times, we get a new matrix $M_n$,
with the same total order as $M$, 
whose entries are greater than or equal to the corresponding entries of $M$,
and such that ${\sum_{i,j}{d_{ij}(M_n)}}=0$ (i.e. $\Gamma(M_n)=\{1,\dots,\ell\}^2$).
Hence, by the above result, we can find an optimal majorant of $M_n$,
which is also an optimal majorant for $M$ since $m_{ij}\leq(M_n)_{ij}$ for every $i,j=1,\dots,\ell$,
and $\tord(M)=\tord(M_n)$.
\end{remark}
\pecettaold{
From the proof of Lemma \ref{20111208:lem1} we have that,
\begin{itemize}
\item 
if $\dd(A)=0$, then $\det_1A\in\mc A$,
\item
Suppose $\dd(A)=1$ and fix $i$. 
Let $(f_1,\dots,f_\ell)\in\mc A^\ell$ be left eigenvector of $\bar A(1)$ of e-value 0
(note: we can assume that $f_1,\dots,f_\ell$ have no common factors,
i.e. the ideal $(f_1,\dots,f_\ell)\subset\mc A$ is not contained in any principal ideal).
Then letting $P$ be the matrix as in the proof (with the $f$'s in row $i$),
$PA$ is strongly non-degenerate,
and we get $f_i\det_1A=(\det_1P)(\det_1A)=\det_1(PA)\in\mc A$, namely $\det_1A\in\mc A/f_i$.
In conclusion, $\det_1A\in\bigcap_{i=1}^\ell\frac{\mc A}{f_i}$.
\item
Suppose $\dd(A)=2$ and fix $i,j$. 
Let $(f_1,\dots,f_\ell)\in\mc A^\ell$ be left eigenvector of $\bar A(1)$ of e-value 0
and let $P$ be the matrix as in the proof (with the $f$'s in row $i$).
Then $PA$ has $\dd(PA)\leq1$.
Let then  $(g_1,\dots,g_\ell)\in\mc A^\ell$ be left eigenvector of $\overline{PA}(1)$ of e-value 0
and let $Q$ be the matrix as in the proof (with the $g$'s in row $j$).
Then $QPA$ is strongly non-degenerate, and 
we get $g_jf_i\det_1A=(\det_1Q)(\det_1P)(\det_1A)=\det_1(QPA)\in\mc A$, namely $\det_1A\in\mc A/f_ig_j$.
In conclusion, $\det_1A\in\bigcap_{i,j=1}^\ell\frac{\mc A}{f_ig_j}$.
\item
In general, if $\dd(A)=r\geq1$, we conclude that there exist
elements $(F^\alpha_1,\dots,F^\alpha_\ell)\in\mc A^\ell,\,\alpha=1,\dots,r$,
each with no common factors,
such that
$\det_1A\in\bigcap_{i_1,\dots,i_r=1}^\ell\frac{\mc A}{F^1_{i_1}\dots F^r_{i_r}}$.
\end{itemize}
}

\begin{theorem}\label{20111215:thm1}
Let $\mc A$ be a unital differential subring of the differential field $\mc K$,
and assume that $\mc A$ is integrally closed.
Let $A\in\Mat_{\ell\times\ell}\mc A((\partial^{-1}))$ be a matrix with $\det A\neq0$.
Then $A$ is invertible in $\Mat_{\ell\times\ell}\mc A((\partial^{-1}))$
if and only if $\det_1A$ is invertible in $\mc A$.
\end{theorem}
\begin{lemma}\label{20111215:lem1}
Let $\mc A$ be an arbitrary unital differential subring of the differential field $\mc K$,
and assume that the matrix $A\in\Mat_{\ell\times\ell}\mc A((\partial^{-1}))$
is strongly non-degenerate.
Then $A$ is invertible in $\Mat_{\ell\times\ell}\mc A((\partial^{-1}))$
if and only if $\det_1A$ is invertible in $\mc A$.
\end{lemma}
\begin{proof}
Let $\{N_i,h_i\}_{i=1}^\ell$ be an optimal majorant for the matrix $A$,
and consider the new matrix
$$
\tilde A=
\left(\begin{array}{ccc}
\partial^{h_1} & & 0 \\
& \ddots & \\
0 & & \partial^{h_\ell}
\end{array}\right)
A
\left(\begin{array}{ccc}
\partial^{-N_1} & & 0 \\
& \ddots & \\
0 & & \partial^{-N_\ell}
\end{array}\right)
\,\in\Mat{}_{\ell\times\ell}\mc A((\partial^{-1}))\,.
$$
Since, by assumption, $A$ is strongly non-degenerate,
the matrix $\bar A(1)=\big(A_{ij;N_j-h_i}\big)_{i,j=1}^\ell$ is non-degenerate.
But this matrix is the leading coefficient of the matrix $\tilde A$.
To conclude, we observe that, 
$\det_1A=\det\bar A(1)$ is invertible in $\mc A$,
if and only if the matrix $\bar A(1)$ is invertible in $\Mat_{\ell\times\ell}\mc A$,
which, by formula \eqref{20111130:eq2} (which works also in the matrix case),
happens if and only if the matrix $\tilde A$ is invertible in $\Mat_{\ell\times\ell}\mc A((\partial^{-1}))$,
which, obviously, 
is the same as saying that the matrix $A$ is invertible in $\Mat_{\ell\times\ell}\mc A((\partial^{-1}))$.
\end{proof}
\begin{proof}[Proof of Theorem \ref{20111215:thm1}]
Since $\mc A$ is integrally closed, by Theorem \ref{20111208:thm2}
we have $\det_1 A\in\mc A$.
Moreover, if the matrix $A$ is invertible in $\Mat_{\ell\times\ell}\mc A((\partial^{-1}))$,
then also $\det_1 A^{-1}\in\mc A$,
therefore $\det_1 A$ is an invertible element of $\mc A$.
Conversely, assume that $\det_1 A$ is invertible in $\mathcal{A}$. 
Let $\mc B\subset\mc K$ be any valuation subring of $\mc K$ containing $\mc A$,
and let $A^{-1}$ be the inverse of $A$ in $\Mat_{\ell\times\ell}\mc K((\partial^{-1}))$.
By Lemma \ref{20111208:lem2}
there exists a matrix $P\in\Mat_{\ell\times\ell}\mc B((\partial^{-1}))$
such that $\det_1P=1$ and $PA$ is stronlgy non-degenerate.
By assumption $\det_1(PA)=\det_1A$ is invertible in $\mc A$ (hence in $\mc B$),
therefore, by Lemma \ref{20111215:lem1}, the matrix $PA$ is invertible in $\Mat_{\ell\times\ell}\mc B((\partial^{-1}))$.
On the other hand, $P$ is product of matrices of the form \eqref{20111214:eq1},
and each such factor is obviously invertible in $\Mat_{\ell\times\ell}\mc B((\partial^{-1}))$.
Hence, $A^{-1}=(PA)^{-1}P\in\Mat_{\ell\times\ell}\mc B((\partial^{-1}))$.
Since this holds for every valuation ring $\mc B\subset\mc K$ containing $\mc A$,
we obtain the claim.
\end{proof}


\appendix

\section{Valuation rings}
\label{sec:app}

Recall that a unital subring $\mc A$ of a field $\mc K$ is called a \emph{valuation ring}
if for any two non-zero elements $a,b\in\mc A$
either $\frac ab\in\mc A$ or $\frac ba\in\mc A$.
Recall also that $\mc A\subset\mc K$ is called \emph{integrally closed} in $\mc K$
if the solutions (in $\mc K$) of every monic polynomial equation with coefficients in $\mc A$
lie in $\mc A$.
The \emph{integral closure} of $\mc A\subset\mc K$ is the minimal subring of $\mc K$
containing $\mc A$ which is integrally closed.
The following fact is well known (see e.g. \cite[Cor.5.22]{AMC69}):
\begin{proposition}\label{20111210:prop1}
The integral closure of $\mc A$ in $\mc K$
is the intersection of all valuation rings of $\mc K$ which contain $\mc A$.
\end{proposition}
The above proposition, and the following result, are used in the proof of Theorem \ref{20111208:thm2}. 
\begin{proposition}\label{20111210:prop2}
Let $\ell\geq2$ be a fixed integer. A unital subring $\mc A$ of the field $\mc K$
is a valuation ring if and only if the following condition holds:
\begin{enumerate}
\item[($A_\ell$)]
for every $\ell$-tuple $(a_1,\dots,a_\ell)\in\mc K^\ell\backslash\{(0,\dots,0)\}$
there exists $i$ such that
$a_i\neq0$ and $\frac{a_1}{a_i},\dots,\frac{a_\ell}{a_i}\in\mc A$.
\end{enumerate}
\end{proposition}
\begin{proof}
First, note that condition $(A_2)$ is the same as the definition of valuation ring.
Clearly, for $\ell\geq3$, condition $(A_\ell)$ implies condition $(A_{\ell-1})$,
by letting $a_\ell=0$.
It remains to prove that condition $(A_2)$ implies condition $(A_\ell)$ for every $\ell$.

Let $(a_1,\dots,a_\ell)\in\mc K^\ell\backslash\{(0,\dots,0)\}$.
If one of the entries is zero, the claim holds by assumption.
Hence, we may assume that $a_1,\dots,a_\ell$ are all non-zero.
Introduce a total order on the set $\{1,\dots,\ell\}$
by letting $j\leq i$ if $\frac{a_j}{a_i}$ lies in $\mc A$, 
and $j=i$ if both $\frac{a_j}{a_i}$ and $\frac{a_i}{a_j}$ lie in $\mc A$.
The transitivity property of $\leq$ follows from the fact that $\mc A$ is a ring,
and $\leq$ is a total order thanks to the assumption $(A_2)$.
Then, letting $i$ be a maximal element, we get the desired result.
%
%
%
\end{proof}



\end{document}